\documentclass[smallcondensed]{elsarticle}
\usepackage{amssymb,amsmath}
\usepackage{color}
\usepackage{bm}

\usepackage{ntheorem}
\newtheorem{thm}{Theorem}[section]

\newtheorem{remark}[thm]{Remark}

\newtheorem*{proof}[thm]{Proof.}

\newcommand{\ds}{\displaystyle}

\newcommand{\mS}{\mathcal S}

\newcommand{\p}{\partial}

\usepackage{amssymb}
\usepackage{bm} 

\usepackage{amsmath}
\usepackage{amscd}
\usepackage{amsfonts}
\journal{ }
\begin{document}

\begin{frontmatter}

\title{Approximate inversion for Abel integral operators of variable exponent and applications to fractional Cauchy problems}

\author[1]{Xiangcheng Zheng}
\ead{zhengxch@math.pku.edu.cn}

\address[1]{School of Mathematical Sciences, Peking University, Beijing 100871, China}

\begin{abstract}
We investigate the variable-exponent Abel integral equations and corresponding fractional Cauchy problems. The main contributions of the work are: (i) We develop an approximate inversion technique to convert variable-exponent Abel integral operators to feasible forms, based on which we analyze the corresponding integral and differential equations; (ii) We prove that the sensitive dependence of the well-posedness of classical Riemann-Liouville fractional differential equations on the initial value could be resolved by adjusting the initial value of the variable exponent; (iii) We prove that the singularity of the solutions to the Riemann-Liouville fractional differential equations could also be eliminated by adjusting the variable exponent and its derivatives at the initial time, which, together with (ii), demonstrates the advantages of introducing the variable exponent. The above findings suggest that the variable-exponent fractional problems may serve as a connection between integer-order and fractional models by adjusting the variable exponent at the initial time.
\end{abstract}

\end{frontmatter}

\section{Introduction}
Abel integral equation of $0<\alpha<1$
$$\mathcal I_t^\alpha u(t):=\int_0^t\frac{u(s)}{(t-s)^\alpha}ds=f(t) $$
has a long history with widely applications and extensive investigations, see e.g., \cite{Gor,Hac} and the references therein. However, for the variable-exponent problems, in which the exponent $\alpha$ may be a function of $s$ or $t$, are rarely studied \cite{LorHar,Sam,Sam93} due to, e.g., the unavailability of the exact solutions. There are some recent progresses on the analysis and numerical approximations to the second kind Volterra integral equations (VIEs) with variable exponents and the variable-exponent time or space fractional models that could be reduced to second kind VIEs via variable substitutions or spectral decompositions \cite{LiaSty,WanJMAA,ZheSINUM}, but the variable-exponent Abel integral problems, which are indeed first kind VIEs with weak singularities, and the corresponding variable-exponent fractional Cauchy problems remain untreated in the literature to our best knowledge.

The critical difficulty of analyzing these problems lies in finding their inversions or equivalent second kind VIEs, which is not encountered in \cite{LiaSty,WanJMAA,ZheSINUM} as the models in these works are naturally second kind VIEs. In this paper we aim to provide a potential means to analyze the variable-exponent Abel integral equations and corresponding fractional differential equations involving the following variable-exponent Abel integral operators
\begin{equation}\label{AbelI}
 \mathcal I^{\alpha(t)}g(t):=\int_0^t\frac{g(s)}{(t-s)^{\alpha(s)}}ds,~~\hat{\mathcal I}^{\alpha(t)}g(t):=\int_0^t\frac{g(s)}{(t-s)^{\alpha(t)}}ds
 \end{equation}
or their variants. The $\alpha(t)$ lies in between $0$ and $1$ and we will specify its range in each section. As applications of introducing the variable exponent, we find that the sensitive dependence of the well-posedness of classical Riemann-Liouville fractional differential equations (R-L FDEs) on the initial value and the initial singularity of their solutions could be resolved by adjusting the initial value of the variable exponent. To be specific, the main contributions of this work are listed as follows:
\begin{itemize}
\item[$\bullet$] We find the approximate inversion operators of these variable-exponent Abel integral operators in the sense that they are reverted as the identity operator (or its multiple) added by a weak-singular integral operator, which takes the form of the second kind VIE. We base on this to analyze the well-posedness and smoothing properties of Abel integral equations and corresponding fractional Cauchy problems involving (\ref{AbelI}) or their variants.

\item[$\bullet$] It is known that the constant-exponent R-L fractional Cauchy problem of order $0<\alpha<1$ with the initial condition $u(0)=u_0$ is well-posed only in the case of $u_0=0$, which may be restrictive in real applications. Instead, in most literature, the fractional initial condition is proposed for this problem, which may be difficult to determine in real applications. Different from the existing results, we prove that the variable-exponent R-L FDEs of order $\alpha(t)$ are well-posed for any $u_0\in\mathbb R$ by setting $\alpha(0)=1$, which resolves the sensitive dependence of the well-posedness of classical R-L FDE on the initial value without using fractional initial conditions.

\item The derivative of the solutions to the classical R-L FDEs of order $0<\alpha<1$ with the zero initial condition is in general unbounded at the initial time. We prove that this singularity could be eliminated by setting $\alpha(0)=1$ and $(\p_t\alpha)(0)=0$ in the variable-exponent R-L FDEs of order $\alpha(t)$, which not only maintains the advantages of the fractional-order operators in characterizing the memory effects, but resolves the initial singularity of the solutions. 
\end{itemize}

Indeed, the conditions on the variable exponent proposed in these items suggest that by setting a integer limit of the variable exponent at the initial time, the properties of the FDEs tend to those of integer-order differential equations. Therefore, the variable-exponent fractional problems may serve as a connection between integer-order and fractional models by adjusting the variable exponent at the initial time.

The rest of the paper is organized as follows: In Section 2 we present approximate inversions of several variable-exponent Abel integral operators, based on which we analyze the well-posedness of variable-exponent Abel integral equations in Section 3. In Section 4, we prove the smoothing properties of their solutions. In Section 5, we analyze the variable-exponent R-L FDEs and discuss the dependence of their well-posedness on the initial values and the singularity issue of their solutions. We finally address conclusions and potential future works in Section 6.
\section{Approximate inversions of variable-exponent Abel integral operators}
We introduce approximate inversion operators of variable-exponent Abel integral operators  in the sense that it reverts an operator to the identity operator (or its multiple) added by an integral operator with a weak-singular kernel, which could be employed to analyze the corresponding integral and differential equations in the rest of the paper. Let $C^m[0,T]$ for $0\leq m\in\mathbb N$ be the space of $m$-th continuously differentiable functions on $[0,T]$, and $L^p(0,T)$ with $1\leq p\leq \infty$ be the space of $p$-th Lebesgue integrable functions on $(0,T)$. All these spaces are equipped with standard norms \cite{Ada}. We also use $C^m(0,T]$ to denote the space of functions that are $m$-th continuously differentiable on $[\varepsilon,T]$ for any $0<\varepsilon\ll 1$.
\subsection{Approximate inversions of $\mathcal I^{\alpha(t)}_t$ and $\hat{\mathcal I}^{\alpha(t)}_t$}
We consider these two operators with $0<\alpha(t)<1$ on $[0,T]$ and $\alpha\in C^1[0,T]$. Note that these conditions imply that $\alpha(t)$ lies in the interior of $[0,1]$ for any $t\in[0,T]$ with lower and upper bounds $0<\alpha_*\leq\alpha^*<1$.
\begin{thm}\label{thmII}
The operators
 $$\hat{\mathcal D}^{\alpha(t)}_t:=\p_t\,\mathcal {\hat I}^{1-\alpha(t)}_t\text{ and } \mathcal D^{\alpha(t)}_t:=\p_t\,\mathcal {I}^{1-\alpha(t)}_t$$ 
 serve as approximate inversion operators of $\mathcal I^{\alpha(t)}_t$ and $\hat{\mathcal I}^{\alpha(t)}_t$, respectively, in the sense that the following relations hold when both sides of the equations exist
\begin{align}\label{rela}
&\hat{\mathcal D}^{\alpha(t)}_t \mathcal I^{\alpha(t)}_tg(t)=\gamma(t)g(t)+\int_0^t\mathcal K(s,t)g(s)ds,\\
&\label{rela2} {\mathcal D}^{\alpha(t)}_t \hat{\mathcal I}^{\alpha(t)}_tg(t)=\gamma(t)g(t)+\int_0^t\mathcal L(s,t)g(s)ds,
\end{align}
where $\Gamma(\cdot)$ and $B(\cdot,\cdot)$ refer to standard Gamma and Beta functions, respectively, 
$$\gamma(t):=\Gamma(\alpha(t))\Gamma(1-\alpha(t))$$
 and the kernels $\mathcal K$ and $\mathcal L$ are defined as
\begin{align*}
&\mathcal K(s,t):=\p_t\big(B(\alpha(t),1-\alpha(s))(t-s)^{\alpha(t)-\alpha(s)}\big),\\
& \mathcal L(s,t):=\p_t\bigg(\int_0^1 \frac{1}{(1-z)^{1-\alpha((t-s)z+s)}z^{\alpha((t-s)z+s)}} dz \bigg),
\end{align*}
satisfying for $0\leq s\leq t\leq T$
\begin{equation}\label{bndK}
|\mathcal K(s,t)|\leq Q(1+|\ln(t-s)|),~~|\mathcal L(s,t)|\leq Q.
\end{equation}
Here $Q$ may depend on $\alpha_*$, $\alpha^*$, $\|\alpha\|_{C^1[0,T]}$ and $T$.
\end{thm}
\begin{proof}
By definitions we apply the Beta integral to obtain
\begin{equation}\label{II} \begin{array}{rl}
\ds \mathcal {\hat I}^{1-\alpha(t)}_t \mathcal I^{\alpha(t)}_t g(t)&\ds=\int_0^t\frac{1}{(t-s)^{1-\alpha(t)}}\int_0^s\frac{g(y)}{(s-y)^{\alpha(y)}} dyds\\[0.15in]
&\ds=\int_0^tg(y)\int_y^t\frac{1}{(t-s)^{1-\alpha(t)}}\frac{1}{(s-y)^{\alpha(y)}} dsdy\\[0.15in]
&\ds=\int_0^tB(\alpha(t),1-\alpha(y))(t-y)^{\alpha(t)-\alpha(y)}g(y)dy.
\end{array} \end{equation}
As $\alpha\in C^1[0,T]$ we have
$$\lim_{y\rightarrow t^-}(t-y)^{\alpha(t)-\alpha(y)}=\lim_{y\rightarrow t^-} e^{(\alpha(t)-\alpha(y))\ln(t-y)}=e^0=1. $$
Indeed, a general estimate on this function reads:
\begin{equation*}
(t-y)^{\alpha(t)-\alpha(y)} \text{ has positive upper and lower bounds for }0\leq y\leq t\leq T.
\end{equation*}
 Then we differentiate (\ref{II}) and apply
$$B(\alpha(t),1-\alpha(t))=\Gamma(\alpha(t))\Gamma(1-\alpha(t)) $$
 to obtain
(\ref{rela}). Direct calculations show that
\begin{equation}\label{DK}
\begin{array}{l}
\ds \mathcal K(s,t)=\big(\p_t B(\alpha(t),1-\alpha(s))\big) (t-s)^{\alpha(t)-\alpha(s)}+ B(\alpha(t),1-\alpha(s))\\
\ds\qquad\qquad\times(t-s)^{\alpha(t)-\alpha(s)} \Big(\p_t\alpha(t)\ln(t-s)+\frac{\alpha(t)-\alpha(s)}{t-s}\Big).
\end{array}
\end{equation}
Since $\alpha\in C^1[0,T]$ and is bounded away from $0$ and $1$, the derivative and difference quotient of $\alpha$ as well as $B(\alpha(t),1-\alpha(s))$ and its partial derivatives are bounded, which leads to the first estimate of (\ref{bndK}). 

To prove (\ref{rela2}) and the second estimate of (\ref{bndK}), we employ the transformation $s=(t-y)z+y$ to obtain 
\begin{equation*}
 \begin{array}{rl}
\ds \mathcal I^{1-\alpha(t)}_t \hat{\mathcal I}^{\alpha(t)}_t g(t)&\ds=\int_0^t\frac{1}{(t-s)^{1-\alpha(s)}}\int_0^s\frac{g(y)}{(s-y)^{\alpha(s)}} dyds\\[0.15in]
&\ds=\int_0^tg(y)\int_y^t\frac{1}{(t-s)^{1-\alpha(s)}}\frac{1}{(s-y)^{\alpha(s)}} dsdy\\[0.15in]
&\ds=\int_0^tg(y)\int_0^1 \frac{1}{(1-z)^{1-\alpha((t-y)z+y)}z^{\alpha((t-y)z+y)}} dz dy.
\end{array} 
\end{equation*}
Differentiating this equation yields
\begin{equation}\label{II3} 
\ds \mathcal D^{\alpha(t)}_t \hat{\mathcal I}^{\alpha(t)}_t g(t)\ds=g(t)\int_0^1 \frac{1}{(1-z)^{1-\alpha(t)}z^{\alpha(t)}} dz +\int_0^tg(y)\mathcal L(y,t) dy.
 \end{equation}
The first right-hand side term of (\ref{II3}) is exactly $\Gamma(\alpha(t))\Gamma(1-\alpha(t))g(t)$, and the kernel of the second right-hand side term could be bounded as
$$\begin{array}{rl}
\ds |\mathcal L(y,t)|&\ds=\bigg|\int_0^1 \frac{\p_t\alpha((t-y)z+y)(\ln(1-z)-\ln z)}{(1-z)^{1-\alpha((t-y)z+y)}z^{\alpha((t-y)z+y)}} dz\bigg|\\[0.15in]
&\ds\leq Q\int_0^1\frac{|\ln(1-z)|+|\ln z|}{(1-z)^{1-\alpha_*}z^{\alpha^*}}dz\leq Q,
\end{array} $$
which completes the proof.
\end{proof}
\subsection{Extensions}
Based on the ideas as above, we could find approximate inversions for other commonly-used variable-exponent integral operators, and we list some representative examples in the following. 

\paragraph{Example 1: Variable-exponent R-L fractional integral operators}
Such operators have attracted increasingly attentions in recent years and is defined for $0<\alpha(t)< 1$ \cite{LorHar,Mog,Sun,Tav}
\begin{align}\label{RLI}
 I_t^{\alpha(t)}g(t):=\int_0^t\frac{1}{\Gamma(\alpha(s))}\frac{g(s)}{(t-s)^{1-\alpha(s)}}ds,\\\label{RLhI}
 \ds \hat I_t^{\alpha(t)}g(t):=\int_0^t\frac{1}{\Gamma(\alpha(t))}\frac{g(s)}{(t-s)^{1-\alpha(t)}}ds.
 \end{align}
 
 \begin{remark}
 Indeed, we may set $\alpha(0)=0$ for (\ref{RLI})--(\ref{RLhI}).
In the variable-exponent Abel integral operators (\ref{AbelI}), setting $\alpha(0)=1$ could make the definitions inappropriate. For instance, we evaluate $\lim_{t\rightarrow 0^+}\hat{\mathcal I}_t^{\alpha(t)}1$ as
$$\lim_{t\rightarrow 0^+}\hat{\mathcal I}_t^{\alpha(t)}1=\lim_{t\rightarrow 0^+}\frac{t^{1-\alpha(t)}}{1-\alpha(t)}=\lim_{t\rightarrow 0^+}\frac{e^{(1-\alpha(t))\ln t}}{1-\alpha(t)}=+\infty. $$
 However, (\ref{RLI})--(\ref{RLhI}) do not encounter this issue if $\alpha(0)=0$ by virtue of the additional Gamma function on the denominator. For instance, we again evaluate $\lim_{t\rightarrow 0^+}\hat{I}_t^{\alpha(t)}1$ as
$$\lim_{t\rightarrow 0^+}\hat{I}_t^{\alpha(t)}1=\lim_{t\rightarrow 0^+}\frac{t^{\alpha(t)}}{\Gamma(1+\alpha(t))}=\lim_{t\rightarrow 0^+}\frac{e^{\alpha(t)\ln t}}{\Gamma(1+\alpha(t))}=1. $$
\end{remark}

 The corresponding  approximate inversion operators are the variable-exponent R-L fractional differential operators $\hat D_t^{\alpha(t)}$ and $ D_t^{\alpha(t)}$, respectively, defined by \cite{LorHar,Mog,Sun,Tav}
\begin{align*}
\hat D_t^{\alpha(t)}g(t):=\p_t\hat I_t^{1-\alpha(t)}g(t)=\p_t\bigg( \int_0^t\frac{1}{\Gamma(1-\alpha(t))}\frac{g(s)}{(t-s)^{\alpha(t)}}ds\bigg),\\
 D_t^{\alpha(t)}g(t):=\p_t\hat I_t^{1-\alpha(t)}g(t)=\p_t\bigg( \int_0^t\frac{1}{\Gamma(1-\alpha(s))}\frac{g(s)}{(t-s)^{\alpha(s)}}ds\bigg),
\end{align*}
which, similar to the derivations as (\ref{II})--(\ref{II3}), satisfy
\begin{align}\label{DIfrac} 
\hat D_t^{\alpha(t)} I_t^{\alpha(t)}g(t)&=g(t)+\int_0^t\p_t\frac{(t-y)^{\alpha(y)-\alpha(t)}}{\Gamma(1-\alpha(t)+\alpha(y))}g(y)dy\\\nonumber
D_t^{\alpha(t)} \hat I_t^{\alpha(t)}g(t)&=g(t)+ \int_0^tg(y)\p_t\int_0^1 \frac{(1-z)^{-\alpha((t-y)z+y)}z^{\alpha((t-y)z+y)-1}}{\Gamma(1-\alpha((t-y)z+y))\Gamma(\alpha((t-y)z+y))} dz dy.
\end{align}
In Section \ref{secRL} we will apply these relations to investigate the variable-exponent R-L FDEs. 

\paragraph{Example 2: A general variable-exponent R-L fractional integral operator}
The definition of this operator is given for a $C^1$ function $0<\alpha(\cdot,\cdot)<1$ \cite{LorHar,Sun}
\begin{equation*} I_t^{\alpha(t,\cdot)}g(t):=\int_0^t\frac{1}{\Gamma(\alpha(t,s))}\frac{g(s)}{(t-s)^{1-\alpha(t,s)}}ds.
 \end{equation*}
 The corresponding  approximate inversion operator is \cite{LorHar,Sun}
\begin{align*}
 D_t^{\alpha(\cdot,t)}g(t):=\p_t I_t^{1-\alpha(\cdot,t)}g(t)=\p_t\bigg( \int_0^t\frac{1}{\Gamma(1-\alpha(s,t))}\frac{g(s)}{(t-s)^{\alpha(s,t)}}ds\bigg),
\end{align*}
which, by the substitution $s=(t-y)z+y$, satisfies
\begin{equation*}\begin{array}{rl}
\ds D_t^{\alpha(\cdot,t)} I_t^{\alpha(t,\cdot)}g(t)&\ds=\p_t\int_0^t\frac{(t-s)^{-\alpha(s,t)}}{\Gamma(1-\alpha(s,t))}\int_0^s\frac{(s-y)^{\alpha(s,y)-1}}{\Gamma(\alpha(s,y))} g(y)dyds\\[0.15in]
&\ds\hspace{-0.15in}=\p_t\int_0^tg(y)\int_y^t \frac{(t-s)^{-\alpha(s,t)}(s-y)^{\alpha(s,y)-1}}{\Gamma(1-\alpha(s,t))\Gamma(\alpha(s,y))}  dsdy\\[0.15in]
&\ds\hspace{-0.15in}=\p_t\int_0^tg(y)\mathcal M(y,t)dzdy=g(t)+\int_0^t\p_t\mathcal M(y,t)g(y)dy
\end{array} \end{equation*}
where 
$$\begin{array}{l}
\ds\mathcal M(y,t):=\int_0^1\frac{(t-y)^{\alpha((t-y)z+y,y)-\alpha((t-y)z+y,t)}}{(1-z)^{\alpha((t-y)z+y,t)}z^{1-\alpha((t-y)z+y,y)}}\\[0.15in]
\ds\qquad\qquad\quad\times \frac{1}{\Gamma(1-\alpha((t-y)z+y,t))\Gamma(\alpha((t-y)z+y,y)) }dz.
\end{array}  $$
Though the kernel $\p_t\mathcal M(y,t)$ looks complicated, it is indeed bounded by $Q(1+|\ln(t-y)|)$ as $\mathcal K$ in Theorem \ref{thmII}. 

\paragraph{Example 3: Tempered variable-exponent R-L fractional integral operators} These kind of operators are defined for $\sigma\geq 0$ and $0<\alpha(t)<1$ \cite{Den,Sab}
\begin{align*}
&{}^\sigma I_t^{\alpha(t)}g(t):=\int_0^t\frac{e^{-\sigma(t-s)}}{\Gamma(\alpha(s))}\frac{g(s)}{(t-s)^{1-\alpha(s)}}ds,\\
&{}^\sigma\hat I_t^{\alpha(t)}g(t):=\int_0^t\frac{e^{-\sigma(t-s)}}{\Gamma(\alpha(t))}\frac{g(s)}{(t-s)^{1-\alpha(t)}}ds.
\end{align*}
By virtue of the semigroup property of the exponential function, the corresponding approximate inversion operators could be given by the tempered variable-exponent R-L fractional differential operators ${}^\sigma \hat D_t^{\alpha(t)}$ and ${}^\sigma D_t^{\alpha(t)}$, respectively, defined by \cite{Den,Sab}
\begin{align*}
&{}^\sigma \hat D_t^{\alpha(t)}g(t):=\p_t{}^\sigma\hat I_t^{1-\alpha(t)}g(t)=\p_t\bigg( \int_0^t\frac{e^{-\sigma(t-s)}}{\Gamma(1-\alpha(t))}\frac{g(s)}{(t-s)^{\alpha(t)}}ds\bigg),\\
&{}^\sigma D_t^{\alpha(t)}g(t):=\p_t{}^\sigma I_t^{1-\alpha(t)}g(t)=\p_t\bigg( \int_0^t\frac{e^{-\sigma(t-s)}}{\Gamma(1-\alpha(s))}\frac{g(s)}{(t-s)^{\alpha(s)}}ds\bigg).
\end{align*}
\section{Well-posedness of variable-exponent Abel integral equations}\label{secAbel}
Based on the previous section, we prove the well-posedness of the variable-exponent Abel integral equations with $0<\alpha(t)<1$
\begin{equation}\label{Abel}
\mathcal I^{\alpha(t)}_tu(t)=f(t),~~t\in [0,T].
\end{equation} 
Its analogous problems that replace $\alpha(s)$ in $\mathcal I^{\alpha(t)}_t$ by $\alpha(t)$ or $\alpha(s,t)$ could be analyzed similarly.

By Theorem \ref{II} and the fact that $\gamma(t)$ is bounded and bounded away from $0$, integral equation (\ref{Abel}) could be transformed as a second kind VIE by applying the operator $\hat{\mathcal D}^{\alpha(t)}_t$ on both sides  
\begin{equation}\label{VIE2}
v(t)+\frac{1}{\gamma(t)}\int_0^t\mathcal K(s,t)v(s)ds=\frac{\hat{\mathcal D}^{\alpha(t)}_tf(t)}{\gamma(t)}.
\end{equation}
Here we use $v(t)$ in (\ref{VIE2}) instead of $u(t)$ for the clarity of the notations in the following proof.
\begin{thm}\label{thmabel}
For $\alpha,f\in C^1[0,T]$ with $f(0)\neq 0$, there exists a unique solution $u\in C(0,T]$ to the variable-exponent Abel integral equation (\ref{Abel}) such that $\lim_{t\rightarrow 0^+}t^{1-\alpha(0)}u(t)$ exists, $|u(0)|=\infty$ and
\begin{equation*}
\|t^{1-\alpha(0)}u\|_{C[0,T]}\leq Q\|f\|_{C^1[0,T]}
\end{equation*}
where $Q$  may depend on $\alpha_*$, $\alpha^*$, $\|\alpha\|_{C^1[0,T]}$ and $T$.

In particular, if $f(0)=0$, then $u\in C[0,T]$ with $u(0)=0$ and the stability estimate
\begin{equation}\label{thme2}
\|u\|_{C[0,T]}\leq Q\|f\|_{C^1[0,T]}.
\end{equation}
\end{thm}
\begin{proof}
Instead of considering (\ref{Abel}) directly, we first investigate (\ref{VIE2}). Define the space $C_{\alpha(0)}(0,T]:=\{v\in C(0,T]:\lim_{t\rightarrow 0^+}t^{1-\alpha(0)}v(t)\text{ exists}\}$ equipped with the norm $\|v\|_{C_{\alpha(0)}(0,T]}:=\|e^{-\lambda t}t^{1-\alpha(0)}v\|_{C[0,T]}=\|e^{-\lambda t}t^{1-\alpha(0)}v\|_{L^\infty(0,T)}$ for some $\lambda>e\approx 2.718$ and a mapping $\mathcal S$ between this space by 
 \begin{equation}\label{lem11}
\mS v(t):=-\frac{1}{\gamma(t)}\int_0^t\mathcal K(s,t)v(s)ds+\frac{\hat{\mathcal D}^{\alpha(t)}_tf(t)}{\gamma(t)}.
\end{equation}
We need firstly to show that $\mS$ is well defined. Neglecting the $\gamma(t)$ on the denominator, which is bounded away from $0$,  direct calculations show that
\begin{equation}\label{Df} \hat{\mathcal D}^{\alpha(t)}_tf=\p_t\int_0^t\frac{f(t-s)}{s^{1-\alpha(t)}}ds=\frac{f(0)}{t^{1-\alpha(t)}}+\int_0^t\frac{\p_tf(t-s)}{s^{1-\alpha(t)}}+f(t-s)\p_t \bigg(\frac{1}{s^{1-\alpha(t)}}\bigg)ds. \end{equation}
As for any $0<\varepsilon\ll 1$ such that $1-\alpha(t)+\varepsilon<1$, the following estimates hold $$\begin{array}{c}
\ds\bigg|\p_t \bigg(\frac{1}{s^{1-\alpha(t)}}\bigg)\bigg|=\bigg|\frac{\p_t\alpha(t)\ln s}{s^{1-\alpha(t)}}\bigg|\leq Q\frac{\ln s}{s^{1-\alpha(t)}}\leq Q\frac{1}{s^{1-\alpha(t)+\varepsilon}},\\[0.15in]
\ds \frac{t^{1-\alpha(0)}}{t^{1-\alpha(t)}}=t^{\alpha(t)-\alpha(0)}=e^{(\alpha(t)-\alpha(0))\ln(t-0)}\leq Q,
\end{array}  $$
we obtain $\hat{\mathcal D}^{\alpha(t)}_tf\in C_{\alpha(0)}(0,T]$. By (\ref{bndK}) and $v\in C_{\alpha(0)}(0,T]$, the first right-hand side term of (\ref{lem11}) also belongs to $ C_{\alpha(0)}(0,T]$. Therefore, we conclude from (\ref{lem11}) that $\mS v\in C_{\alpha(0)}(0,T]$ and thus $\mS$ is well defined.

To prove the well-posedness of (\ref{lem11}), it remains to show the contractivity of the integral operator in the first right-hand side term of (\ref{lem11}) under the norm $\|\cdot\|_{C_{\alpha(0)}(0,T]}$. We apply the splitting $t^{1-\alpha(0)}=s^{1-\alpha(0)}+(t^{1-\alpha(0)}-s^{1-\alpha(0)})$ to bound this term by
\begin{equation}\label{lem13} \begin{array}{l}
\ds \bigg\|\frac{1}{\gamma(t)}\int_0^t\mathcal K(s,t)v(s)ds\bigg\|_{C_{\alpha(0)}(0,T]}\\
\ds\leq Q\bigg\|\int_0^t|\mathcal K(s,t)e^{-\lambda (t-s)} s^{1-\alpha(0)}e^{-\lambda s}v(s)|ds\\[0.15in]
\ds\quad+\int_0^t\frac{|t^{1-\alpha(0)}-s^{1-\alpha(0)}|}{s^{1-\alpha(0)}}|\mathcal K(s,t)e^{-\lambda (t-s)}s^{1-\alpha(0)}e^{-\lambda s}v(s)|ds\bigg\|_{L^\infty(0,T)}\\
\ds \leq Q\|t^{1-\alpha(0)}e^{-\lambda t}v\|_{L^\infty(0,T)}\bigg\|\int_0^t|\mathcal K(s,t)|e^{-\lambda (t-s)} ds\\[0.15in]
\ds\quad+\int_0^t\frac{(t-s)^{1-\alpha(0)}}{s^{1-\alpha(0)}}|\mathcal K(s,t)|e^{-\lambda (t-s)}ds\bigg\|_{L^\infty(0,T)}
\end{array} \end{equation}
By (\ref{bndK}) and the relations \cite{GraRyz} \begin{equation}\label{asym}\begin{array}{c}
\ds \int_0^t(t-s)^{\mu-1}e^{-\lambda s}ds=B(\mu,1)t^{\mu} {}_1F_1(1;1+\mu,-\lambda t),~~\mu>0,\\
\ds {}_1F_1(1;\mu, x)=\frac{\Gamma(\mu)}{\Gamma(\mu-1)}\frac{1}{x}\bigg(1+\mathcal O\bigg(\frac{1}{x}\bigg)\bigg),~~x\rightarrow-\infty,~~\mu\geq 1,
\end{array} \end{equation}
where ${}_1F_1$ refers to the hypergeometric Kummer function, we have
$$\begin{array}{l}
\ds\bigg|\int_0^t|\mathcal K(s,t)|e^{-\lambda (t-s)} ds\bigg|\leq Q\int_0^t(t-s)^{-\varepsilon}e^{-\lambda (t-s)}ds\\
\ds\qquad\qquad=\frac{Q}{\lambda^{1-\varepsilon}}\int_0^{\lambda t}y^{-\varepsilon}e^{-y}dy\leq\frac{Q}{\lambda^{1-\varepsilon}},~~0<\varepsilon\ll 1,\\[0.15in]
\ds \int_0^t\frac{(t-s)^{1-\alpha(0)}}{s^{1-\alpha(0)}}|\mathcal K(s,t)|e^{-\lambda (t-s)}ds\leq Q\int_0^t\frac{e^{-\lambda (t-s)}}{s^{1-\alpha(0)}}ds\\[0.15in]
\ds\qquad=\frac{QB(\alpha(0),1)}{\lambda^{\alpha(0)}}(\lambda t)^{\alpha(0)}{}_1F_1(1;1+\alpha(0),-\lambda t)\leq \frac{Q}{\lambda^{\alpha(0)}}.
\end{array}  $$
Here we used the asymptotic property of ${}_1F_1$ in (\ref{asym}) to achieve the boundedness of $(\lambda t)^{\alpha(0)}{}_1F_1(1;1+\alpha(0),-\lambda t)$ for large $\lambda t$. We invoke these estimates in (\ref{lem13}) to obtain
\begin{equation*} \begin{array}{l}
\ds \bigg\|\frac{1}{\gamma(t)}\int_0^t\mathcal K(s,t)v(s)ds\bigg\|_{C_{\alpha(0)}(0,T]}\leq  Q\bigg(\frac{1}{\lambda^{1-\varepsilon}}+\frac{1}{\lambda^{\alpha(0)}}\bigg)\|v\|_{C_{\alpha(0)}(0,T]}.
\end{array} \end{equation*}
Then for $\lambda$ large enough, the integral operator in the first right-hand side term of (\ref{lem11}) is a contraction under the norm $\|\cdot\|_{C_{\alpha(0)}(0,T]}$, which implies that there exists a unique solution in $C_{\alpha(0)}(0,T]$ to (\ref{VIE2}) with the stability estimate
\begin{equation*}
\|v\|_{C_{\alpha(0)}(0,T]}\leq Q\|\hat{\mathcal D}^{\alpha(t)}_tf\|_{C_{\alpha(0)}(0,T]}\leq Q\|f\|_{C^1[0,T]}.
\end{equation*}

Now we turn to analyze (\ref{Abel}). We rewrite (\ref{VIE2}) back to its original form as
\begin{equation}\label{DI}
\hat{\mathcal D}^{\alpha(t)}_t \mathcal I^{\alpha(t)}_tv(t)=\hat{\mathcal D}^{\alpha(t)}_tf(t). 
\end{equation}
By (\ref{II}) and $v\in C_{\alpha(0)}(0,T]$, which implies 
$$\lim_{t\rightarrow 0^+}\mathcal {\hat I}^{1-\alpha(t)}_t \mathcal I^{\alpha(t)}_t v(t)=0, $$
we integrate (\ref{DI}) from $0$ to $t$ to get
$$\mathcal {\hat I}^{1-\alpha(t)}_t \big(\mathcal I^{\alpha(t)}_t v(t)-f(t)\big)=0. $$
By the continuity of $\mathcal I^{\alpha(t)}_t v(t)-f(t)$, we could prove by contradiction that $\mathcal I^{\alpha(t)}_t v(t)=f(t)$, that is, $v(t)$ serves as a solution to model (\ref{Abel}). In particular, we pass the limit $t\rightarrow 0^+$ for (\ref{VIE2}) and apply (\ref{Df}) to find that $|v(0)|=\infty$. The uniqueness of the the solutions to (\ref{Abel}) in $C_{\alpha(0)}(0,T]$ follows from that of (\ref{VIE2}).

Finally, we observe from (\ref{Df}) that if $f(0)=0$, then the right-hand side of (\ref{Df}) and thus $\hat{\mathcal D}^{\alpha(t)}_t f$ is continuous on $[0,T]$, and the above derivations could be performed in $C[0,T]$ directly without introducing the weight function $t^{1-\alpha(0)}$. In particular, we pass the limit $t\rightarrow 0^+$ for (\ref{VIE2}) and apply (\ref{Df}) to find that $v(0)=0$. Thus we complete the proof of the theorem.
\end{proof}

\section{Smoothing properties of variable-exponent Abel integral equations}\label{secreg}
In this section we prove the estimate of $\p_tu(t)$ where $u(t)$ is the solution to the variable-exponent Abel integral equation (\ref{Abel}). 
\begin{thm}\label{thmreg}
Suppose $\alpha$, $f\in C^2[0,T]$. Then the solution $u$ to the variable-exponent Abel integral equation (\ref{Abel}) belongs to $C^1(0,T]$ with the estimate
\begin{equation}\label{thmee1}
|\p_tu(t)|\leq Q\|f\|_{C^2[0,T]}t^{\alpha(0)-2},~~t\in (0,T].
\end{equation}
Here $Q$ may depend on $\alpha_*$, $\alpha^*$, $\|\alpha\|_{C^2[0,T]}$ and $T$.

If in addition $f(0)=0$ but $(\p_tf)(0)\neq 0$, $u\in C[0,T]\cap C^1(0,T]$ with the estimate
\begin{equation*}
|\p_tu(t)|\leq Q\|f\|_{C^2[0,T]}t^{\alpha(0)-1},~~t\in (0,T].
\end{equation*}
Otherwise, if $f(0)=(\p_tf)(0)=0$, $u\in  C^1[0,T]$ with the estimate
\begin{equation*}
|\p_tu(t)|\leq Q\|f\|_{C^2[0,T]},~~t\in [0,T].
\end{equation*}
\end{thm}
\begin{proof}
We replace $v(t)$ by $u(t)$ in (\ref{VIE2}), multiply $t$ on both sides of the resulting equation, and split $t=s+(t-s)$ to obtain
\begin{equation}\label{reg1}
tu(t)+\frac{1}{\gamma(t)}\int_0^t\mathcal K(s,t)su(s)ds+\frac{1}{\gamma(t)}\int_0^t\mathcal K(s,t)(t-s)u(s)ds=\frac{t\hat{\mathcal D}^{\alpha(t)}_tf(t)}{\gamma(t)}.
\end{equation}
By (\ref{DK}) and
$$\begin{array}{l}
\ds\p_s \big(B(\alpha(t),1-\alpha(s))(t-s)^{\alpha(t)-\alpha(s)}\big)=\big(\p_s B(\alpha(t),1-\alpha(s))\big)(t-s)^{\alpha(t)-\alpha(s)}\\
\ds\quad+B(\alpha(t),1-\alpha(s))(t-s)^{\alpha(t)-\alpha(s)}\bigg(-\p_s\alpha(s)\ln(t-s)-\frac{\alpha(t)-\alpha(s)}{t-s}\bigg),
\end{array} $$
we obtain
$$\begin{array}{l}
\ds\mathcal K(s,t)=-\p_s \big(B(\alpha(t),1-\alpha(s))(t-s)^{\alpha(t)-\alpha(s)}\big)\\[0.05in]
\ds\qquad\qquad+\big((\p_t+\p_s) B(\alpha(t),1-\alpha(s))\big) (t-s)^{\alpha(t)-\alpha(s)}\\[0.05in]
\ds\qquad\qquad +B(\alpha(t),1-\alpha(s))(t-s)^{\alpha(t)-\alpha(s)} (\p_t\alpha(t)-\p_s\alpha(s))\ln(t-s)\\[0.05in]
\ds\qquad\quad\,=:-\p_s \big(B(\alpha(t),1-\alpha(s))(t-s)^{\alpha(t)-\alpha(s)}\big)+\mathfrak{K}(s,t),
\end{array}  $$
and we could apply $\alpha\in C^2[0,T]$ to find 
\begin{equation*}
\lim_{s\rightarrow t^-}\mathfrak{K}(s,t)=0,~~|\p_t\mathfrak{K}(s,t)|\leq Q(1+|\ln(t-s)|).
\end{equation*}
  We invoke these in the second left-hand side term of (\ref{reg1}) and apply integration by parts to obtain
\begin{equation}\label{rege25}
\begin{array}{l}
\ds \frac{1}{\gamma(t)}\int_0^t\mathcal K(s,t)su(s)ds\\
\ds\qquad=-tu(t)+\frac{1}{\gamma(t)}\int_0^tB(\alpha(t),1-\alpha(s))(t-s)^{\alpha(t)-\alpha(s)}\p_s(su(s))ds\\[0.15in]
\ds\qquad\qquad+\int_0^t\frac{\mathfrak{K}(s,t)}{\gamma(t)}su(s)ds.
\end{array}
\end{equation}
Then we differentiate this equation to obtain
\begin{equation*}
\begin{array}{l}
\ds \p_t\bigg(\frac{1}{\gamma(t)}\int_0^t\mathcal K(s,t)su(s)ds\bigg)\\[0.15in]
\ds\qquad=\int_0^t \p_t\bigg(\frac{1}{\gamma(t)}B(\alpha(t),1-\alpha(s))(t-s)^{\alpha(t)-\alpha(s)}\bigg)\p_s(su(s))ds\\[0.15in]
\ds\qquad\qquad+\int_0^t\p_t\bigg(\frac{\mathfrak{K}(s,t)}{\gamma(t)}\bigg)su(s)ds.
\end{array}
\end{equation*}
We use this relation to differentiate (\ref{reg1}) as follows
\begin{equation}\label{reg2}
\begin{array}{l}
\ds \hspace{-0.15in}\p_t(tu(t))+\int_0^t \p_t\bigg(\frac{1}{\gamma(t)}B(\alpha(t),1-\alpha(s))(t-s)^{\alpha(t)-\alpha(s)}\bigg)\p_s(su(s))ds\\[0.15in]
\ds\hspace{-0.1in}+\int_0^t\p_t\bigg(\frac{\mathfrak{K}(s,t)}{\gamma(t)}\bigg)su(s)ds+\int_0^t\p_t\bigg(\frac{\mathcal K(s,t)(t-s)}{\gamma(t)}\bigg)u(s)ds=\p_t\bigg(\frac{t\hat{\mathcal D}^{\alpha(t)}_tf(t)}{\gamma(t)}\bigg).
\end{array}
\end{equation}
By (\ref{DK}) we have
$$\bigg|\p_t\bigg(\frac{\mathcal K(s,t)(t-s)}{\gamma(t)}\bigg)\bigg|\leq Q(1+\ln(t-s)).$$
Thus, we apply Theorem \ref{thmabel} to bound the third and the forth terms on the left-hand side of (\ref{reg2}) by $Q\|f\|_{C^1[0,T]}$. We calculate $t\hat{\mathcal D}^{\alpha(t)}_tf(t)$ by (\ref{Df}) as
\begin{equation*}
 t\hat{\mathcal D}^{\alpha(t)}_tf=f(0)t^{\alpha(t)}+t\int_0^t\frac{\p_tf(t-s)}{s^{1-\alpha(t)}}+f(t-s)\p_t \bigg(\frac{1}{s^{1-\alpha(t)}}\bigg)ds,
\end{equation*}
the derivative of which could be bounded as
$\p_t(t\hat{\mathcal D}^{\alpha(t)}_tf(t))\leq Q\|f\|_{C^2[0,T]}t^{\alpha(0)-1}.$ We incorporate these estimates in (\ref{reg2}) to obtain for some $0<\varepsilon<1$
$$\begin{array}{rl}
\ds|\p_t(tu(t))|&\ds\leq Q\int_0^t(1+|\ln(t-s)|)|\p_s(su(s))|ds+Q\|f\|_{C^2[0,T]}t^{\alpha(0)-1}\\[0.15in]
&\ds\leq Q\int_0^t\frac{|\p_s(su(s))|}{(t-s)^\varepsilon}ds+Q\|f\|_{C^2[0,T]}t^{\alpha(0)-1}.
\end{array}  $$
Then an application of the weak singular Gronwall inequality, see e.g., \cite[Theorem 1.2]{Web}, yields
$$|\p_t(tu(t))|\leq Q\|f\|_{C^2[0,T]}t^{\alpha(0)-1}.
 $$
 As $\p_t(tu(t))=u(t)+t\p_tu(t)$, we immediately obtain from Theorem \ref{thmabel} that 
 $$|t\p_tu(t)|\leq Q\|f\|_{C^2[0,T]}t^{\alpha(0)-1},$$
  which proves (\ref{thmee1}). 
  
  If $f(0)=0$, then Theorem \ref{thmabel} implies that model (\ref{Abel}) admits a unique continuous solution on $[0,T]$ satisfying $u(0)=0$ and (\ref{thme2}). Thus we could directly differentiate (\ref{VIE2}) and perform estimates by similar techniques as above without multiplying $t$ on both sides. In particular, differentiating the second right-hand side term of (\ref{Df}) yields
  $$\p_t\int_0^t\frac{\p_tf(t-s)}{s^{1-\alpha(t)}}ds=\frac{(\p_tf)(0)}{t^{1-\alpha(t)}}+ \int_0^t\p_t\bigg(\frac{\p_tf(t-s)}{s^{1-\alpha(t)}}\bigg)ds.$$
If $(\p_tf)(0)\neq 0$, then $\p_tu$ has the singularity of $t^{\alpha(t)-1}\leq Qt^{\alpha(0)-1}$. Otherwise, $\p_t u$ is continuous on $[0,T]$, which completes the proof.
\end{proof}

\section{Analysis of R-L fractional differential equations}\label{secRL}
Classical R-L FDE of order $0<\alpha<1$ reads \cite{DieFor,KilSri}
\begin{equation}\label{RLFDE0}
D^{\alpha}_tu(t)=h(t),~~t\in (0,T];~~u(0)=u_0.
\end{equation}
the solution of which could be analytically obtained as
\begin{equation}\label{uh}
 u(t)=I^{\alpha}_th(t).
 \end{equation}
Note that if $h\in C[0,T]$, taking $t\rightarrow 0^+$ on both sides of this equation leads to $u(0)=0$. Thus, $u_0$ must be $0$ in order to ensure the well-posedness of (\ref{RLFDE0}), which is restricted and extremely sensitive to the initial value, that is, any perturbation or noise on the initial value may lead to the ill-posedness of model (\ref{RLFDE0}). In most literature, the equation in (\ref{RLFDE0}) is usually equipped with the fractional initial condition, which may be difficult to determine in real applications.

Another observation is that if we differentiate (\ref{uh}) we obtain 
$$\p_tu(t)=\p_t I^{\alpha}_th(t)=\frac{t^{\alpha-1}}{\Gamma(\alpha)}h(0)+I^{\alpha}_t \p_t  h(t),$$
which indicates that the solution has a singular derivative if $h(0)\neq 0$. We will show that in variable exponent problems, this singularity could be removed by adjusting the variable exponent.
 
In this section we address these issues to study the following variable-exponent R-L FDE 
\begin{equation}\label{RLFDE}
D^{\alpha(t)}_tu(t)=h(t),~~t\in (0,T];~~u(0)=u_0.
\end{equation}
The analysis for its analogues problems, which replace $I^{1-\alpha(t)}_t$ in (\ref{RLFDE}) by $\hat I^{1-\alpha(t)}_t$ or $I^{1-\alpha(t,\cdot)}_t$,
could be carried out similarly. We discuss this model for two cases: 
\begin{itemize}
\item[(\textbf{i})] $0<\alpha(t)<1$ on $[0,T]$;

\item[(\textbf{ii})] $0<\alpha(t)<1$ on $(0,T]$ and $\alpha(0)=1$.
\end{itemize}
We will draw conclusions that under the condition (\textbf{i}), (\ref{RLFDE}) encounters the same sensitivity issue on the initial condition and the solution has a singular derivative as its constant-exponent analogue (\ref{RLFDE0}), while under the condition (\textbf{ii}), model (\ref{RLFDE}) is well-posed for any $u_0\in\mathbb R$ and, if in addition $(\p_t\alpha)(0)=0$, the solution has a continuous derivative, which resolves the sensitivity issue of the well-posedness of R-L FDEs on the initial conditions and the initial singularities of the solutions to FDEs.

\subsection{Model (\ref{RLFDE}) with condition (\textbf{i})}
We aim to seek its solution in $C[0,T]$ and formally integrate (\ref{RLFDE}) from $0$ to $t$ to obtain 
\begin{equation}\label{vie1}
I^{1-\alpha(t)}_tu(t)=\int_0^th(s)ds+c_0.
\end{equation}
Here $c_0$ is a constant that will be determined later.
As $0<\alpha(t)<1$ on $[0,T]$, the results in Sections \ref{secAbel}--\ref{secreg} could be employed since models (\ref{vie1}) and (\ref{Abel}) are almost equivalent (except for an additional Gamma function in $I^{1-\alpha(t)}_t$ that does not affect the analysis) if 
$$f(t)=\int_0^th(s)ds+c_0.$$
If $c_0\neq 0$, then $f(0)\neq 0$. By Theorem \ref{thmabel}, the solution $u$ has the property $|u(0)|=\infty$, which could not satisfy the initial condition in (\ref{RLFDE}). Therefore, we must set $c_0=0$ and then Theorem \ref{thmabel} implies that there exists a unique continuous solution to (\ref{vie1}) with $u(0)=0$, that is, $u_0$ must be $0$ in order that model (\ref{RLFDE}) is well-posed, which encounters the same sensitivity issue on the initial condition as its constant-exponent analogue (\ref{RLFDE0}). In this case ($u_0=0$), we could apply Theorem \ref{thmreg} to conclude that if $h(0)\neq 0$, then the solution $u$ to (\ref{vie1}) and thus (\ref{RLFDE}) has a initial singularity of $\mathcal O(t^{\alpha(0)-1})$ as the constant-exponent case.

\subsection{Model (\ref{RLFDE}) with condition (\textbf{ii})}
By condition (\textbf{ii}), $\alpha(t)$ still has a positive lower bound $\alpha_*$ but is no longer bounded away from $1$. We apply $\hat D_t^{1-\alpha(t)}$ on both sides of (\ref{vie1}) and employ (\ref{DIfrac}) to obtain
\begin{equation}\label{vie4}
u(t)+\int_0^t\p_t\frac{(t-y)^{\alpha(t)-\alpha(y)}}{\Gamma(1+\alpha(t)-\alpha(y))}u(y)dy=\hat D^{1-\alpha(t)}_t\bigg(\int_0^th(s)ds+c_0\bigg).
\end{equation}
 We base on this equation to analyze the well-posedness and smoothing properties of variable-exponent R-L FDE (\ref{RLFDE}) in the following theorem.
\begin{thm}
Suppose $\alpha\in C^1[0,T]$, $h\in C[0,T]$
and the condition (\textbf{ii}) holds. Then the variable-exponent R-L FDE (\ref{RLFDE}) admits a unique solution in $C[0,T]$ for any $u_0\in\mathbb R$ with the stability estimate
$$\|u\|_{C[0,T]}\leq Q\big(\|h\|_{C[0,T]}+|u_0|\big).$$
Here $Q$ may depend on $\alpha_*$, $\|\alpha\|_{C^1[0,T]}$ and $T$.

If $\alpha\in C^2[0,T]$ with $(\p_t\alpha)(0)=0$, $h\in C^1[0,T]$ and the condition (\textbf{ii}) holds, then $u\in C^1[0,T]$ with the stability estimate
$$\|u\|_{C^1[0,T]}\leq Q\big(\|h\|_{C^1[0,T]}+|u_0|\big).$$
Here $Q$ may depend on $\alpha_*$, $\|\alpha\|_{C^2[0,T]}$ and $T$.
  \end{thm}
\begin{proof}
We apply the transformation $s\rightarrow t-s$ to obtain
\begin{equation}\label{dh} \begin{array}{l}
\ds \hat D_t^{1-\alpha(t)}\int_0^th(s)ds=\p_t \hat I^{\alpha(t)}_t\int_0^th(s)ds=\p_t\int_0^t\frac{s^{\alpha(t)-1}}{\Gamma(\alpha(t))}\int_0^{t-s}h(y)dyds\\[0.15in]
\ds\qquad=\int_0^t\p_t\bigg(\frac{s^{\alpha(t)-1}}{\Gamma(\alpha(t))}\bigg)\int_0^{t-s}h(y)dyds+\int_0^t\frac{s^{\alpha(t)-1}}{\Gamma(\alpha(t))}h(t-s)ds.
\end{array} \end{equation}
As $\alpha(t)\geq \alpha_*>0 $, we have
$$\bigg|\p_t\bigg(\frac{s^{\alpha(t)-1}}{\Gamma(\alpha(t))}\bigg)\bigg|=\bigg|\frac{s^{\alpha(t)-1}}{\Gamma(\alpha(t))}\bigg(\p_t\alpha(t)\ln s-\frac{\p_t\Gamma(\alpha(t))}{\Gamma(\alpha(t))}\bigg)\bigg|\leq Qs^{\alpha_*-1}|\ln s|. $$
Therefore, the right-hand side of (\ref{dh}) is continuous on $[0,T]$. We could also evaluate $\hat D_t^{1-\alpha(t)}c_0$ exactly as follows
\begin{equation}\label{dc0}
\begin{array}{rl}
\ds \hat D_t^{1-\alpha(t)}c_0&\ds=c_0\p_t\int_0^t\frac{1}{\Gamma(\alpha(t))}\frac{1}{(t-s)^{1-\alpha(t)}}ds=c_0\p_t\bigg(\frac{t^{\alpha(t)}}{\Gamma(1+\alpha(t))}\bigg)\\[0.15in]
&\ds = c_0\bigg(\frac{t^{\alpha(t)-1}}{\Gamma(\alpha(t))}+\frac{(\p_t\alpha(t)) t^{\alpha(t)}\ln t}{\Gamma(1+\alpha(t))}+t^{\alpha(t)}\p_t\bigg(\frac{1}{\Gamma(1+\alpha(t))}\bigg)\bigg).
\end{array}
\end{equation}
As $\alpha(0)=1$ and $\alpha\in C^1[0,T]$, $\lim_{t\rightarrow 0^+}t^{\alpha(t)-1}=\lim_{t\rightarrow 0^+}e^{(\alpha(t)-\alpha(0))\ln t}=1$. Similar to the estimate of $\mathcal K(s,t)$, we have $|\p_t(t-y)^{\alpha(t)-\alpha(y)}|\leq Q(1+|\ln (t-y)|)$. Then by similar proofs as Theorem \ref{thmabel}, integral equation (\ref{vie4}) admits a unique solution in $C[0,T]$ with stability estimate $\|u\|_{C[0,T]}\leq Q\big(\|h\|_{C[0,T]}+c_0\big)$. Passing the limit $t\rightarrow 0^+$ on both sides of (\ref{vie4}) and using the expressions (\ref{dh}) and (\ref{dc0}) we obtain
\begin{equation*}
u(0)=c_0.
\end{equation*}
Therefore, if we set $c_0=u_0$, then the solution $u$ to model (\ref{vie4}) satisfies $u(0)=u_0$.

Finally we need to recover the differential equation (\ref{RLFDE}) from the integral equation (\ref{vie4}), the original form of which is
\begin{equation*}
\p_t \hat I^{\alpha(t)}_t\bigg( I^{1-\alpha(t)}_tu(t)-\int_0^th(s)ds-c_0\bigg)=0.
\end{equation*}
As the content in $(\cdots)$ is continuous, we integrate this equation from $0$ to $t$ to obtain
\begin{equation*}
 \hat I^{\alpha(t)}_t\bigg( I^{1-\alpha(t)}_tu(t)-\int_0^th(s)ds-c_0\bigg)=0,
\end{equation*}
which, by contraction, immediately leads to
\begin{equation*}
  I^{1-\alpha(t)}_tu(t)-\int_0^th(s)ds-c_0=0.
\end{equation*}
As the second and the third left-hand side terms of this equation are differentiable, so does the first term. Thus we differentiate this equation to obtain (\ref{RLFDE}), that is, the solution to the integral equation (\ref{vie4}) also solves (\ref{RLFDE}). The uniqueness of the solutions to (\ref{RLFDE}) in $C[0,T]$ follows from that of (\ref{vie4}).

To bound $\p_tu(t)$, we may differentiate (\ref{vie4}) and then perform estimates following the proof of Theorem \ref{thmreg}. In particular, when we apply the integration by parts for the second left-hand side term of (\ref{vie4}) as (\ref{rege25}), the same term as the first right-hand side term of (\ref{dc0}) with $c_0=u_0$ will be generated, and they cancel each other.
If we differentiate the numerator of the second right-hand side term of (\ref{dc0}), we will encounter the factor $$(\p_t\alpha(t))\alpha(t)t^{\alpha(t)-1}\ln t.$$
In order to make this term continuous, we require $\alpha(0)=1$, $(\p_t\alpha)(0)=0$ and and $\alpha\in C^2[0,T]$ as assumed in the theorem.
 The rest of the proof is omitted as it is similar to that of Theorem \ref{thmreg} and thus we complete the proof.
\end{proof}

\section{Concluding remarks}
We develop an approximate inversion technique of variable-exponent Abel integral operators, and analyze the corresponding integral and differential equations. In particular, we prove that the sensitive dependence of the well-posedness of classical R-L FDEs on the initial value and the singularity of the solutions could be resolved by adjusting the initial value of the variable exponent, which demonstrates its advantages. The proposed approximate inversion technique provides a potential means to convert the intricate variable-exponent integral and differential problems to feasible forms and the derived results suggest that the variable-exponent fractional problems may serve as a connection between integer-order and fractional models by adjusting the variable exponent at the initial time.

There are other potential applications of this approximate inversion method. For instance, since the variable-exponent Abel integral equations and corresponding fractional differential equations have been converted as second-kind VIEs, traditional treatments for the second-kind VIEs \cite{Bru} could be employed to prove the high-order regularity of the solutions and to develop and analyze numerical schemes. We will investigate these topics in the near future.

\section*{Acknowledgements}
This work was partially funded by the  International Postdoctoral Exchange Fellowship Program (Talent-Introduction Program) YJ20210019 and by the China Postdoctoral Science Foundation 2021TQ0017.


\end{document}